\documentclass[11pt,reqno]{amsart}
\usepackage{amsmath,amsfonts,amssymb}
\usepackage{hyperref}
\usepackage{amsthm}
\usepackage{bibentry}
\usepackage{xcolor}
\usepackage{quiver}
\usepackage{enumerate}
\numberwithin{equation}{section}

\newtheorem{theorem}{Theorem}[section]
\newtheorem{thm}{Theorem}   
\newtheorem{lemma}[theorem]{Lemma}
\newtheorem{proposition}[theorem]{Proposition}
\newtheorem{corollary}[theorem]{Corollary}
\newtheorem{example}[theorem]{Example}
\newtheorem{definition}[theorem]{Definition}
\newtheorem{remark}[theorem]{Remark}

\newcommand{\ab}{^{ab}}
\newcommand{\g}{\mathfrak{g}}

\newcommand{\A}{\mathcal{A}}
\newcommand{\B}{\mathcal{B}}
\newcommand{\D}{\mathcal{D}}
\newcommand{\F}{\mathcal{F}}
\newcommand{\G}{\mathcal{G}}
\newcommand{\HH}{\mathcal{H}}
\newcommand{\R}{\mathbb{R}}
\newcommand{\Z}{\mathbb{Z}}
\newcommand{\T}{\mathbb{T}}
\renewcommand{\S}{\mathbb{S}}
\newcommand{\X}{\ensuremath{\mathfrak{X}}}
\newcommand{\Lie}{\mathcal{L}}

\newcommand{\I}{\mathrm{Im}}

\newcommand{\rank}{\mathrm{rank}} 
\newcommand{\id}{\mathrm{Id}} 
\newcommand{\dd}{\mathrm{d}}
\newcommand{\pr}{\mathrm{pr}}
\newcommand{\reg}{\mathrm{reg}}
\begin{document}
\title{Abelianization of Lie algebroids and Lie groupoids}

\author{Shuyu Xiao}
\address{Department of Mathematics, University of Illinois at Urbana-Champaign, 1409 W. Green Street, Urbana, IL 61801 USA}
\email{shuyux2@illinois.edu}
\thanks{Work partially supported by NSF grant DMS-2303586.}

\begin{abstract}
We investigate the abelianization of a Lie algebroid and provide a necessary and sufficient condition for its existence. We also study the abelianization of groupoids and provide sufficient conditions for its existence in the smooth category and a necessary and sufficient condition for its existence in the diffeological category.
\end{abstract}

\maketitle
\section{Introduction}
A Lie algebroid is called \emph{abelian} if its isotropy Lie algebras are all abelian. Given any Lie algebroid $\A$, one can wonder whether there exists an abelian Lie algebroid $\A^{\ab}$ satisfying the universal property of an abelian object. We refer to $\A^{\ab}$ as the \emph{abelianization} of $\A$ (see Section \ref{sec:alg} for the precise definition). 

The notion of abelianization was first introduced in \cite{contreras2018genus} in connection with the so-called genus integration of $\A$. The authors of \cite{contreras2018genus} observed that for a transitive Lie algebroid, the abelianization always exists. They also provide examples of non-transitive Lie algebroids for which the abelianization exists, as well as examples for which it does not exist.

In this paper we introduce a new definition of abelianization. In the transitive case, it coincides with the definition introduced in \cite{contreras2018genus}. Our definition has a sheaf-like nature and is more tractable than the one in \cite{contreras2018genus}. In order to extend the results from \cite{contreras2018genus} to non-transitive Lie algebroids, one can attempt to study the genus integration as a singular object, since the genus integration is a fiberwise quotient of the Weinstein groupoid. The sheaf-like definition allows us to study the local behaviour and to use tools such as sheaves of Lie-Rinehart algebras \cite{villatoro2021sheaveslierinehartalgebras}. Moreover, in various setups Lie algebroids must be treated as sheaf-like objects. For instance, in the holomorphic category, where the bracket is only defined on the sheaf of holomorphic sections of a vector bundle, our definition still makes sense, while the one from \cite{contreras2018genus} does not.

Recall that for any Lie algebroid $\A\to M$, the kernel of the anchor at $x\in M$,
\[ \g_x := \ker \rho_x \subset \A_x, \]
forms a Lie algebra known as the \emph{isotropy Lie algebra} at $x$. We denote the family of isotropy Lie algebras as
\[ \g_M := \bigcup_{x\in M} \g_x \subset \A. \]
Note that, in general, the dimensions of the isotropy Lie algebras $\g_x$ vary with $x,$ so $\g_M$ is not a vector subbundle. One can also define the commutator or derived bundle
\[ [\g_M, \g_M]:=\bigcup_{x\in M} [\g_x,\g_x]\subset\A. \]
which, again, may fail to be a vector subbundle. Our first main result can be stated as follows.

\begin{thm}
\label{thm:main:1}
A Lie algebroid $\A$ has an abelianization if and only if the closure $\overline{[\g_M, \g_M]} \subset \A$ is a vector subbundle.
\end{thm}

Using this result, we study some properties of the abelianization and, in particular, we show that it behaves in a functorial manner relative to pullbacks.
\smallskip

Next, we turn to groupoids. A groupoid is called \emph{abelian} if its isotropy groups are all abelian. As before, we define the \emph{abelianization of a groupoid} to be an abelian groupoid satisfying the universal property of an abelian object (see Section \ref{sec:gp} for the precise definition). The abelianization of a groupoid depends on the category where one works. Here we study the existence and properties of this construction in the smooth and in the diffeological categories.

In general, as already observed in \cite{contreras2018genus}, a Lie groupoid might not admit an abelianization in the smooth category. And when it does, the smooth abelianization might differ from the set-theoretical abelianization. Just as for Lie algebroids, the abelianization of groupoids is closely related to its isotropies, but it is more subtle as the commutator bundle can behave very differently. We have obtained the following partial results concerning abelianization of Lie groupoids.
\begin{proposition}
	Given a Lie groupoid $\G$, if the closure of the commutator $\overline{(\G_M,\G_M)}$ is a subgroupoid, then $\G$ has abelianization $\G/\overline{(\G_M,\G_M)}$. 	
\end{proposition}
\begin{proposition}
	Given a Lie groupoid $\G$, if the fiberwise closure $\overline{(\G_M,\G_M)}^s$ is a closed submanifold, then $\G$ has abelianization $\G/\overline{(\G_M,\G_M)}^s$. 	
\end{proposition}

On the other hand, diffeologies provide a more flexible setup to study differential geometry of singular spaces. For example, it has been used in solving integration problems of  Lie algebroids as in \cite{villatoro2023integrability}, of singular Lie subalgebroids \cite{singinte} and of singular foliations \cite{garmendia2022integration}. As a nice middle ground between topological and smooth spaces, we have the following result about abelianization of diffeological groupoids.
\begin{thm}
	A (locally subductive) diffeological groupoid $\G$ has a (locally subductive) abelianization $\G/(\G_M,\G_M)$.
\end{thm}
\subsection*{Acknowledgements.} The author would like to thank Rui Loja Fernandes for his help and supervision. The author would also like to thank Ivan Contreras for his feedback on writing the paper. Finally, the author would like to thank the referees for their helpful remarks and suggestions on the revision.

\section{Abelianization of Lie algebroids}
\label{sec:alg}
\subsection{Definition of abelianization}
\subsubsection{Background on Lie algebroids}
We recall that a \textbf{Lie algebroid} over $M$ is a triple $(\A,\ \rho, [\cdot,\cdot])$, where
	\begin{itemize}
	 	\item $\A\to M$ is a vector bundle;
	 	\item $\rho:\A\rightarrow TM$ is a bundle map, called the \emph{anchor map};
	 	\item $[\cdot,\cdot]$ is a Lie bracket on the space of sections $\Gamma(\A)$
	\end{itemize} 
such that the Leibniz condition
	\[ [s,fs']=f[s,s']+(\Lie_{\rho(s)}f)s'
	\] 
is satisfied for all $s,s' \in \Gamma(\A)$ and $f\in C^{\infty}(M)$.
	
We often denote the Lie algebroid by $\A$ or $\A\rightarrow M$. The image of $\rho$ is an integrable distribution $T\F\subset TM$. We call the (singular) foliation integrating this distribution the \emph{foliation induced by $\A$} and the leaves of this foliation the \emph{leaves of $\A$}. When $\rho(\A)=TM$, we say that $\A$ is \emph{transitive}.  If $\rho(\A)$ has constant rank, we say that $\A$ is \emph{regular}.  If $\rho=0$, we call $\A$ a \emph{bundle of Lie algebras}. 

\begin{remark}
    Let us take the opportunity to clarify some terminology here. 
    \begin{itemize}
        \item By a \emph{family of Lie algebras} over a manifold $M$ we mean simply a collection of Lie algebras parametrized by $M$.
        \item As we just defined above, a \emph{bundle of Lie algebras} is a Lie algebroid with zero anchor. One can think of a bundle of Lie algebras as a vector bundle equipped with a smoothly varying family of Lie brackets on the fibers.
        \item A \emph{Lie algebra bundle} is essentially a ``locally trivial'' bundle of Lie algebras. To be more precise, it is a fiber bundle with typical fiber a Lie algebra, and an atlas of local trivializations preserving the Lie algebra structures on the fibers.
    \end{itemize}
\end{remark}
For $x\in M$, we have $\g_x:=\ker\,\rho_x$ the \emph{isotropy Lie algebra} at $x\in M$.
As already mentioned in the Introduction, \[\g_M:=\ker(\rho)=\bigcup_{x\in M}\g_x\] is in general not of constant rank, and thus just a family of Lie algebras. We call it - somewhat misleadingly - the \emph{isotropy bundle} of $\A$. Note that when $\A$ is regular, $\g_M$ is actually a subbundle of $\A$ and thus a bundle of Lie algebras. Furthermore, if $\A$ is transitive then it is even a Lie algebra bundle. The \emph{commutator bundle} is defined as
\[ [\g_M, \g_M]:=\bigcup_{x\in M} [\g_x,\g_x]\subset\A \]
and is in general not a vector bundle, even when $\A$ is regular. When $\A$ is transitive however, $[\g_M,\g_M]$ is smooth.

\begin{definition}
    A Lie algebroid $\A$ is called \emph{abelian} if all of its isotropy Lie algebras are abelian.
\end{definition}

Notice that $\A$ is abelian if and only if $[\g_M,\g_M]=0$.

Let $\A\rightarrow M$ and $\B\rightarrow N$ be Lie algebroids, and consider a bundle map 
		\[\begin{tikzcd}[ampersand replacement=\&]
			\A \arrow[r, "F"]\arrow[d] \& \mathcal{B}\arrow[d]\\
			M\arrow[r, "f"]\& N
		\end{tikzcd}\]
If $f$ is a diffeomorphism, then one has an induced map at the level of sections $\tilde{F}:\Gamma(\A)\to\Gamma(\B)$ and one says that $\phi=(F,f)$ is a Lie algebroid morphism if:
\begin{enumerate}[(i)]
    \item $\dd f\circ \rho_{\A}=\rho_{\B}\circ F$;
    \item $\tilde{F}([s,s']_\A)=[\tilde{F}(s),\tilde{F}(s')]_\B$ for any $s,s'\in\Gamma(\A)$.
\end{enumerate}
In general, there is no induced map at the level of sections. For $s\in\Gamma(\A)$ we write $F(s)$ for the section $F\circ s\in\Gamma(f^*\A)$ and we replace (ii) by the following condition:
\begin{enumerate}
    \item[(ii)'] for any sections $s,s'\in\Gamma(\A)$ such that $F(s)=\sum_i a_i f^*{s_i}$ and $F(s')=\sum_j b_j f^*{s_j}$, one has
\[ F([s,s']_{\A})=\sum_{i,j}a_ib_if^*[s_i,s_j]_{\B}+\sum_j(\Lie_{\rho(s)}b_i)f^*(s_j)+\sum_i(\Lie_{\rho(s')}a_i)f^*(s_i). \]
\end{enumerate}
An alternative way to define algebroid morphisms is by using the algebroid de Rham differential on $\A$-forms. One defines $k$-forms to be the sections $\Omega^k(A):=\Gamma(\wedge^k A^*)$ and defines 
\[
\dd_\A:\Omega^k(\A)\to\Omega^{k+1}(\A) 
\]
by 
\begin{align*}
\dd_A\omega(s_0,\dots,s_k):=\sum_i (-1)^i & \Lie_{\rho(s_i)}\omega(s_0,\dots,\widehat{s_i},\dots,s_k)+\\
&\sum_{i<j}(-1)^{i+j}\omega([s_i,s_j],s_0,\dots,\widehat{s_i},\dots,\widehat{s_j},\dots,s_k).
\end{align*}
A bundle map $\phi:\A\to\B$, as above, induces a pull-back map between such forms, and it is a Lie algebroid morphism if and only if
\[ \phi^*\dd_\B=\dd_\A \phi^*. \]
\subsubsection{Abelianization}
Our definition of the abelianization of a Lie algebroid is as follows.

\begin{definition}
	An \textbf{abelianization} of a Lie algebroid $\mathcal{A}\rightarrow M$ consists of
	\begin{enumerate}[(i)]
		\item an abelian Lie algebroid $\mathcal{A}^{\ab}\rightarrow M$, and
		\item a surjective morphism $p:\,\mathcal{A}\rightarrow\mathcal{A}^{\ab}$ covering the identity,
	\end{enumerate}
	such that for any open subset $U\subset M$, any abelian Lie algebroid $\mathcal{B}\rightarrow N$ and morphism $\phi:\mathcal{A}|_U\rightarrow \mathcal{B}$, there is a unique morphism $\tilde{\phi}$ such that the following diagram commutes:
	\[\begin{tikzcd}[ampersand replacement=\&]
			\mathcal{A}|_U \arrow[r, "\phi"]\arrow[d,"p"] \& \mathcal{B}. \\
			\mathcal{A}^{\ab}|_U\arrow[dashed]{ur}[swap]{\exists ! \tilde{\phi}}
    \end{tikzcd}\]
\end{definition}

Clearly, if the abelianization of $\A$ exists, it is unique up to isomorphisms. We will call $\A\ab$ the abelianization of $\A$ when it does not cause confusion. We will denote the anchor of $\A\ab$ by $\rho\ab$ and its bracket by $[\cdot,\cdot]_{ab}$. We will use $(\g\ab)_x$ to denote the isotropy Lie algebras of $\A\ab$ and $\g\ab_x$ to denote the abelianization of the isotropy Lie algebras $\g_x$. In general, these two Lie algebras are different.

\begin{remark}
\label{rem:definition:sheaf:like}
Our definition of abelianization is different from the original one in \cite{contreras2018genus} -- from now on we will refer to their definition as the \emph{CF-abelianization}. The difference is that our definition has a sheaf-like flavor, whereas the CF-abelianization only requires the ``global'' universal property for $U=M$. It follows immediately that our definition is stronger than that of the CF-abelianization. Additionally, if the abelianization exists then it coincides with the CF-abelianization. In particular, the two definitions are equivalent in the case of \emph{transitive} Lie algebroids since these always admit an abelianization (this follows from Theorem \ref{th:main}). Thus all the results in \cite{contreras2018genus} still hold for our definition.
\end{remark}

\begin{example}[Lie algebras]
	Let $\A=\g$ be a Lie algebra. Then its abelianization is
	\[ \g\ab=\g/[\g,\g]. \]
\end{example}

\begin{example}[Transitive Lie algebroids \cite{contreras2018genus}]
	Let $\A\rightarrow M$ be a transitive Lie algebroid. Its isotropy bundle $\g_M$ is then a vector subbundle of $\A$ and in fact an ideal. It is proved in \cite{contreras2018genus} that the abelianization of $\A$ is the quotient algebroid
	\[ \A^{\ab}=\A/[\g_M,\g_M]. \] 
	In particular, $(\g\ab)_x\simeq \g_x\ab$ and $[\g_x,\g_x]\subset\ker(p_x)$, where $p:\A\to\A\ab$ is the quotient map.
\end{example}

In the next sections we aim to investigate the existence of abelianization of general Lie algebroids, beginning with bundles of Lie algebras, progressing to the regular ones and finally the general case. 
\subsection{Bundles of Lie algebras}

Bundle of Lie algebras, even when admitting an abelianization, can already have different properties from the transitive case. 

\begin{example}[\cite{contreras2018genus}]\label{ex:basic}
	Consider $\A=\R^2\times\R\rightarrow \R$ equipped with the bracket 
	\[ [e_1,e_2]_x=xe_1,\] 
	where $e_1$ and $e_2$ denote the constant sections $x\mapsto (1,0,x)$ and $x\mapsto (0,1,x)$. This bundle of Lie algebras has abelianization the bundle of abelian Lie algebras
	 \[ \A\ab=\R\times\R\rightarrow\R, \]
	 Notice that at $x=0$, the isotropy of $\A$ is $\g_0=\g\ab_0\simeq\R^2$, so abelian, but $\A\ab$ has smaller isotropy, namely $(\g\ab)_0=\R$.
\end{example}

In general, a bundle of Lie algebras may not admit an abelianization. Before looking at more examples, let us first take a look at some basic property of an abelianization. The following lemma was proved in \cite{contreras2018genus}. It follows from the fact that $p$ is a surjective morphism covering the identity and that $\A\ab$ is abelian, and is thus easily seen to hold also for our definition of abelianization.
 \begin{lemma}
 	If $\A\ab$ is the abelianization of $\A$, then $\I(\rho)=\I(\rho\ab)$ and $[\g_M,\g_M]\subset\ker p$.
 \end{lemma}
Hence, $\A$ and $\A\ab$ have the same foliation. Since $p:\A\to\A\ab$ has closed kernel, we also deduce that: 

 \begin{corollary}\label{lem:ker:comm}
 	If $\A\ab$ is the abelianization of $\A$, then $\overline{[\g_M,\g_M]}\subset\ker p$.
 \end{corollary}
 
\begin{example}[Bundles of Lie algebras with no abelianization]
	Let us replace the bracket in Example \ref{ex:basic} by 
	\[ [e_1,e_2]_x:=f(x)e_1, \] 
	where $f$ vanishes on an open interval $I=(a,b)\subsetneq\R$. Then the resulting bundle no longer admits an abelianization. Indeed, a potential abelianization $\A\ab$ must have $\rank(\ker p)\geq 1$. But then 
	\[ (\A\ab)_I\subsetneq (\A_I)\ab\simeq \R^2\times I, \] 
	a contradiction (cf.~Remark \ref{rem:definition:sheaf:like}). 
	\smallskip 

	One can also consider the bundle of Lie algebras $\A=\R^2\times\R^2\rightarrow \R^2$ equipped with the bracket
	\[ [e_1,e_2]_{(x,y)}:=xe_1+ye_2.\] 

 The commutator Lie subalgebra is a one dimensional subspace everywhere except for the origin. To be more specific, this one dimensional subspace is essentially ``radiating around the origin''. Thus although $[g_0,g_0]=0$ is of lower dimension, $\overline{[\g_M,\g_M]}|_0\simeq \R^2$ is actually of higher dimension. By Lemma \ref{lem:ker:comm}, a potential abelianization would be $\A\ab=\{0\}\times\R^2$. However, on the open set $O=\R^2\backslash\{0\}$, the quotient map $\A|_O\to \A|_O/[\A|_O,\A|_O]$ is then not able to be recovered since the latter bundle has rank $1$.
 
\end{example}

The previous example can be summarized by saying that if on an open set $O\subsetneq M$ the rank of $[\g_M,\g_M]$ is lower than the highest rank
of $\overline{[\g_M,\g_M]}$, then the abelianization cannot exists since one would obtain
\[ (\A\ab)|_O\subsetneq (\A|_O)\ab. \]
This lead us to the following characterization of the bundles of Lie algebras that admit an abelianization.

\begin{proposition}
    \label{prop:Lie:algebra:bundle}
	A bundle of Lie algebras $\g_M$ has an abelianization if and only if $\overline{[\g_M,\g_M]}$ is a subbundle of $\g_M$. In this case, 
	\[ (\g_M)^{\ab}=\g_M/\overline{[\g_M,\g_M]}. \]
\end{proposition}

\begin{proof} $\quad$\smallskip
	
$(\Leftarrow)$ Let $\g_M$ be a bundle of Lie algebras. If $\overline{[\g_M,\g_M]}$ is a subbundle of $\g_M$, then clearly $\g_M/\overline{[\g_M,\g_M]}$ is an abelian bundle of Lie algebras. Since the quotient map is a surjective morphism covering the identity, we only need to show that it satisfies the universal property. So let $U\subset M$ be an open set, $\B\rightarrow N$ an abelian Lie algebroid and $\phi:(\g_M)|_U\rightarrow\B$ some morphism. Since $\ker \phi$ is closed in $(\g_M)|_U$ and $[\g_M,\g_M]|_U\subset\ker \phi$, we must have $\overline{[\g_M,\g_M]}|_U\subset\ker(\phi)$. It follows that the universal property holds.
\smallskip
	
$(\Rightarrow)$ Suppose $\overline{[\g_M,\g_M]}$ is not a subbundle and $\g_M$ has an abelianization $p:\g_M\to(\g_M)\ab$. Then $\overline{[\g_M,\g_M]}\subset\ker p$ and if we let $n$ be the rank of $\ker p$ we claim that:
\begin{itemize}
\item there exists an open $O\subset M$ such that $(\overline{[\g_M,\g_M]})|_O$ has constant rank $k<n$.
\end{itemize}
Indeed, since the rank of $\overline{[\g_M,\g_M]}_x$ changes upper semi-continuously with $x$, the set 
\[ \{x\in M:\rank_x\overline{[\g_M,\g_M]}=n\}=\{x\in M: \overline{[\g_M,\g_M]}_x=\ker p_x\} \]
is closed. So its complement is a non-empty, open subset $U\subset M$. On the other hand, the sets where rank $\overline{[\g_M,\g_M]}$ is constant $<n$ are finite and their union is $U$, so they cannot all have empty interior.
Let $\B\to O$ be a subbundle of $(\g_M)|_O$ complementary to $(\overline{[\g_M,\g_M]})|_O$, which we view as a trivial Lie algebra bundle. The projection $\phi:(\g_M)|_O\to \B$ is a morphism of Lie algebroids which clearly does not factor through $p:(\g_M)|_O\to(\g_M)\ab|_O$, contradicting the universal property of the abelianization.
\end{proof}

Before moving to more general cases, we prove the following technical but important fact.
\begin{lemma}
\label{lem:dense}
Let $\A$ be a Lie algebroid with $\g_M$ the family of its isotropy Lie algebras. If $\overline{[\g_M,\g_M]}$ is a subbundle, then there exists a dense, open, saturated subset $O\subset M$ such that $[\g_M,\g_M]|_O=\overline{[\g_M,\g_M]}|_O$.
\end{lemma}
\begin{proof}
First assume that $\g_M$ is a bundle of Lie algebras.
Consider the sets
\[ S_k=\{x\in M\mid\dim([\g_M,\g_M]_x)=k\}. \]
and let $n=\max\{k\mid S_k\neq\emptyset\}$. Then since the rank of $[\g_M,\g_M]_x$ changes lower semi-continuously with $x$, we have that $O:=S_n$ is open in $M$ and we have \[ [\g_M,\g_M]|_{O}=\overline{[\g_M,\g_M]}|_{O}. \]
Also, since $\overline{[\g_M,\g_M]}$ is a subbundle, it follows that $n=\rank \overline{[\g_M,\g_M]}$. 
We claim that $\overline{O}=M$. Indeed, if this fails then we can repeat the argument replacing $M$ by $M'=M-\overline{O}$. We obtain sets $S'_k$, with $k<n$, and $n'=\max\{k\mid S'_k\neq\emptyset\}<n$ such that $n'=\rank \overline{[\g_M,\g_M]}$, contradicting that $\overline{[\g_M,\g_M]}$ is a subbundle.\\\\
Now consider a non-regular Lie algebroid, for which $\g_M$ is not a subbundle. Let $U_k:=\{x\in M: \exists \text{ open $V\ni x$ with } \rank(\rho_y)=k \text{ } \forall y\in V\}$ be the set of regular points of rank $k$. Note that the rank of the anchor cannot drop locally, i.e. every $x\in M$ has a neighborhood $U$ such that 
$\rank(\rho_x)\le \rank(\rho_y)$ for all $y\in U$. This implies that each $U_k$ is open. Then by the previous result we see that there exist $O_k\subset U_k$ such that 
\[ [\g_M,\g_M]|_{O_k}=\overline{[\g_M,\g_M]}^{U_k}|_{O_k}=\overline{[\g_M,\g_M]}|_{O_k},\]
where the first closure is within $\A|_{U_k}$. Since the rank takes only a finite number of values, we see that $O:=\cup_k O_k\subset M$ is the desired open dense set.\\
The fact that $O$ is saturated follows from the fact that the sets $S_k$ are saturated. To see this, let $x$ and $y$ belong to the same leaf of $\A$. One can choose a compactly supported section $s\in\Gamma(A)$ such that the time-1 flow of the vector field $\rho(s)$ satisfies
\[ \varphi^1_{\rho(s)}(x)=y.\]
Then the time-1 flow of the section $s$ (see, e.g., \cite{crainic2004integrability}) is a Lie algebroid automorphism which maps $\g_x$ to $\g_y$ and hence also $[\g_M,\g_M]_x$ to $[\g_M,\g_M]_y$. Hence, if $x\in S_k$ we must have $y\in S_k$, so $S_k$ is saturated.
\end{proof}
\begin{remark}\label{re:localreg}
    Given a Lie algebroid $\A\to M$, its regular points $M_\reg:=\cup_k U_k\subset M$ form a dense open subset.
\end{remark}
This allows us to prove one half of Theorem \ref{thm:main:1}. 
\begin{proposition}
    \label{prop:half:main:thm}
    Let $\A$ be a Lie algebroid with isotropy $\g_M$ and assume that $\overline{[\g_M, \g_M]} \subset \A$ is a vector subbundle. Then $\A$ has abelianization the quotient
    \[ p:\A\to \A/\overline{[\g_M, \g_M]}. \]
\end{proposition}

\begin{proof}
    We claim that $\overline{[\g_M, \g_M]}$  is an ideal in $\A$, i.e., that
    \[ s\in\Gamma(\A),\ \xi\in \Gamma(\overline{[\g_M, \g_M]})\quad \Longrightarrow \quad [s,\xi]\in \Gamma(\overline{[\g_M, \g_M]}). \]
    To see this, take $\xi\in \Gamma(\overline{[\g_M, \g_M]})$, then for each $O_k$ as in Lemma \ref{lem:dense}, there are sections $\xi_i\in\Gamma(\g_M|_{O_k})$ and real numbers $a_{ij}$ such that
    \[ \xi|_{O_k}=\sum_{i<j}a_{ij}[\xi_i,\xi_j]. \]
    It follows that if $s\in\Gamma(\A)$, one has
    \[ [s,\xi]|_{O_k}=\sum_{i<j}a_{ij}\left([[s,\xi_i],\xi_j]+[\xi_i,[s,\xi_j]]\right)\in [\g_M,\g_M]|_{O_k}. \]
    Note that $\cup_kU_k$ is dense and open in $M$, and so is $\cup_kO_k$.
    Therefore, we must have $[s,\xi]\in \Gamma(\overline{[\g_M, \g_M]})$, and the claim follows.

    Since $\overline{[\g_M, \g_M]}$  is an ideal in $\A$, it follows that there is a unique Lie algebroid structure such that quotient map
    \[ p:\A\to \A/\overline{[\g_M, \g_M]}\]
    is a morphism of Lie algebroids. Moreover, this quotient has abelian isotropy $\g_M/\overline{[\g_M, \g_M]}$, so we only need to check that the universal property holds.

    Let $U\subset M$ be an open subset, $\mathcal{B}\rightarrow N$ an abelian algebroid and $\phi:\mathcal{A}|_U\rightarrow \mathcal{B}$ an algebroid morphism. Applying Lemma \ref{lem:dense} again, $U\cap O$ is an open dense subset of $U$ where  
    \[ [\g_M,\g_M]|_{O\cap U}=\overline{[\g_M,\g_M]}|_{O\cap U}. \]
    It follows that $\overline{[\g_M,\g_M]}|_U\subset\ker\phi$, so there is a unique morphism $\tilde{\phi}$ such that the following diagram commutes
	\[\begin{tikzcd}[ampersand replacement=\&]
			\mathcal{A}|_U \arrow[r, "\phi"]\arrow[d,"p"] \& \mathcal{B}. \\
			\mathcal{A}^{\ab}|_U\arrow[dashed]{ur}[swap]{\exists ! \tilde{\phi}}
    \end{tikzcd}\]
    so the universal property holds.
\end{proof}
\subsection{Regular Lie algebroids}
We now consider the existence of an abelianization for an arbitrary regular Lie algebroid. Note that a Lie algebroid $\A\to M$ is regular if and only if its isotropy $\g_M\subset \A$ is a subbundle. Therefore, in this case, if $\{s_1,\dots,s_n\}$ is a basis of local sections of $\g_M$ over an open set $U$, one has
\[ [s_i,s_j](x)=[s_i(x),s_j(x)],\quad \forall x\in U. \]
It follows also that $[\g_M,\g_M]_x$ is generated by $[s_i,s_j](x)$.
\subsubsection{Decomposition of regular Lie algebroids} 
In the sequel we will use the fact that one can recover a regular Lie algebroid from its foliation, isotropy bundle and certain extra data concerning the relation between them. We recall briefly how this works and refer to, e.g., to \cite{MK87} for details.

Let $\A$ be a regular Lie algebroid. A choice of a splitting of the short exact sequence defined by the anchor
	\[\begin{tikzcd}[ampersand replacement=\&]
			0 \arrow[r] \& \g_M \arrow[r] \& \A \arrow[r, "\rho"] \& T\F \arrow[r] \arrow[l,dashrightarrow, bend left,"\sigma"] \& 0
    \end{tikzcd}\]
allows to identify $\A$ with $T\F\oplus\g_M$ so that the anchor becomes the projection on $T\F$. On the other hand, the Lie bracket becomes
\begin{equation}
    \label{eq:Lie:bracket:regular}
    [(X,\xi),(Y,\eta)]_\A=([X,Y],[\xi,\eta]_{\g_M}+\nabla_X \eta-\nabla_Y \xi+\Omega(X,Y)), 
\end{equation}
where
\begin{itemize}
	\item $\nabla$ is the $T\F$-connection on $\g_M$ defined by
    \[ \nabla_X \xi:=[\sigma(X),\xi]_{\A}; \]
	\item $\Omega\in\Omega^2(T\F,\g_M)$ is the curvature form of the splitting given by
    \[ \Omega(X,Y):=[\sigma(X),\sigma(Y)]_\A - \sigma([X,Y]_{\A}).\] 
\end{itemize}
It is easy to check that under this isomorphism the Jacobi identity for $[\cdot,\cdot]_\A$ amounts to the following set of identities:
\begin{align}
	& \nabla_X[\xi,\eta]_{\g_M}=[\nabla_X\xi,\eta]_{\g_M}+[\xi,\nabla_X\eta]_{\g_M},\notag \\
	& [\Omega(X,Y),\xi]_{\g_M}=\nabla_X\nabla_Y\xi-\nabla_Y\nabla_X\xi-\nabla_{[X,Y]}\xi, \label{eq:Jacobi:regular}\\
    & \bigodot_{X,Y,Z}\Big(\Omega([X,Y],Z)+\nabla_X(\Omega(Y,Z))\Big)=0, \notag\
\end{align}
where $\xi,\eta\in\Gamma(\g_M)$, $X,Y,Z\in\Gamma(T\F)$, and the symbol $\odot$ denotes cyclic summation.

The converse also holds. Given a foliation $\F$ of $M$, a bundle of Lie algebras $\g_M\to M$, a $T\F$-connection $\nabla$ and a $\g_M$-valued 2-form $\Omega$ satisfying identities \eqref{eq:Jacobi:regular}, then one obtains a Lie algebroid structure on $T\F\oplus\g_M$ with Lie bracket \eqref{eq:Lie:bracket:regular} and anchor $\rho=\pr_{T\F}$. We denote this Lie algebroid by $T\F\ltimes\g_M$. The previous discussion shows that one has the following simple proposition.

\begin{proposition}
    Any regular Lie algebroid is isomorphic to $T\F\ltimes\g_M$ for some quadruple $(\g_M, T\F, \nabla, \Omega)$ satisfying \eqref{eq:Jacobi:regular}. 
\end{proposition}

We will also use the following notations. Given a collections of subspaces
\[ E=\bigcup_{x\in M} E_x\subset\g_M, \]
which is not necessarily a subbundle (e.g., $E=[\g_M,\g_M]$), we still denote the subspace of sections which take values in $E$ by
\[
\Gamma(E):=\{s\in\Gamma(\g_M): s(x)\in E_x \textrm{ for all }x\in M\}. 
\] 
Also, given a $T\F$-connection $\nabla$ on $\g_M$ we will say that $\nabla$ preserves $E$ if for every $X\in\Gamma(T\F)$ one has
\[ 
s\in\Gamma(E)\quad\Longrightarrow\quad \nabla_X s\in\Gamma(E).
\]
The following result will be useful in the sequel. 
\begin{lemma}
    \label{lem:connection:commutator}
	Given a bundle of Lie algebras $\g_M$ and a foliation $\F$ on M, any $T\F$-connection $\nabla$ on $\g_M$ satisfying \eqref{eq:Jacobi:regular}, preserves $[\g_M,\g_M]$. Moreover, $\nabla$ also preserves $\overline{[\g_M,\g_M]}$ provided the latter is a subbundle.
\end{lemma}
\begin{proof}
	The first part is immediate from the first item in criterion (\ref{eq:Jacobi:regular}). The second part follows immediately from Lemma \ref{lem:dense}.
\end{proof}

\subsubsection{Morphisms between regular Lie algebroids}
Consider a surjective morphism of regular Lie algebroids $\phi:\A_1\to\A_2$ covering the identity. Then they share the same foliation $\F$ and we have a commutative diagram with exact rows
  \[\begin{tikzcd}
		0 & \g_M^1 & \A_1 & T\F & 0 \\
		0 & \g_M^2 & \A_2 & T\F &0
		\arrow[from=1-1, to=1-2]
		\arrow[from=1-2, to=1-3]
		\arrow["\rho_1", from=1-3, to=1-4]
		\arrow[from=1-4, to=1-5]
		\arrow[from=2-1, to=2-2]
		\arrow[from=2-2, to=2-3]
		\arrow["\rho_2", from=2-3, to=2-4]
		\arrow[from=2-4, to=2-5]
		\arrow["\phi|_{\g_M^1}",from=1-2, to=2-2]
		\arrow["\phi",from=1-3, to=2-3]
		\arrow[equal,"\id",from=1-4, to=2-4]
        \arrow[dashrightarrow, bend left,"\sigma_1" near start,from=1-4, to=1-3]
        \arrow[dashrightarrow, bend left,"\sigma_2" near start,from=2-4, to=2-3]
  \end{tikzcd}\]
If we choose choose a splitting $\sigma_1:T\F\to \A_1$ of the anchor of $\A_1$, we obtain a splitting $\sigma_2:=\phi\circ\sigma_1:T\F\to \A_2$ of the anchor of $\A_2$. These splittings give identifications
\[ \A_i\simeq T\F\ltimes\g_M^i. \]
We conclude that to specify the morphism $\phi:\A_1\to \A_2$ amounts to specifying
the following data:
	\begin{itemize}
		\item $T\F$-connections $\nabla^i$ on $\g_M^i$ and 2-forms $\Omega^i\in\Omega^2(M,\g_M^i)$ satisfying \eqref{eq:Jacobi:regular}; 
        \item a morphism of Lie algebra bundles $\phi:\g_M^1\rightarrow\g_M^2$ covering the identity compatible with connections $\nabla^i$ and forms $\Omega^i$, i.e., satisfying
	\begin{equation}
		    \label{eq:compatible:morphism}
            \phi\circ\nabla^1=\nabla^2\circ \phi,\quad
			\phi\circ\Omega^1=\Omega^2.
	\end{equation}
    \end{itemize}

In particular, note that the compatibility condition \eqref{eq:compatible:morphism} implies that $\nabla^1$ preserves the kernel of $\phi$ (which may fail to be a subbundle):
\[ s\in\Gamma(\ker(\phi))\quad\Longrightarrow\quad \nabla^1_X s\in\Gamma(\ker(\phi)). \]

Notice also that the connection $\nabla^2$ and the 2-form $\Omega^2$ are completely determined by $\nabla^1$ and $\Omega^1$.This can be stated as follows.

\begin{proposition}
    \label{prop:extend}
    Let $\phi:\g_M^1\to\g_M^2$ be a surjective morphism of Lie algebra bundles and let $\nabla^1$ and $\Omega^1$ be a $T\F$-connection and 2-form on $\g_M^1$ satisfying \eqref{eq:compatible:morphism}. If $\nabla^1$ preserves $\ker\phi$ then:
    \begin{enumerate}[(i)]
        \item there exist a unique $T\F$-connection $\nabla^2$ and a unique 2-form $\Omega^2$ on $\g_M^2$ satisfying \eqref{eq:compatible:morphism}, and
        \item  $\phi$ extends to a surjective Lie algebroid morphism
    \[ (\id,\phi):T\F\ltimes\g_M^1\to T\F\ltimes\g_M^2. \]
    \end{enumerate}
\end{proposition}

\begin{proof}
    Let $\sigma:\g_M^2\to \g_M^1$ be a splitting of $\phi:\g_M^1\to\g_M^2$. Then the expression
    \[ \nabla^2_X\xi:=\phi(\nabla_X^1(\sigma\circ\xi)), \]
    defines a $T\F$-connection on $\g_M^2$. Since $\nabla^1$ preserves $\ker\phi$, if one sets $\Omega^2:=\phi\circ\Omega^1$, one checks easily that the pair $(\nabla^2,\Omega^2)$ satisfies \eqref{eq:Jacobi:regular} and that \eqref{eq:compatible:morphism} holds. Both items should now be obvious.
\end{proof}

This can also be restated as follows.

\begin{corollary}
\label{co:conn}
    Let $\phi:\A_1\to\A_2$ be a surjective morphism, covering the identity, between regular Lie algebroids. A choice of splitting of the anchor of $\A_1$ determines isomorphisms $\A_i\simeq T\F\ltimes\g_M^i$, such that $\phi$ becomes
    \[ \phi=(\id,\phi|_\g): T\F\ltimes\g_M^1\to  T\F\ltimes\g_M^2. \]
\end{corollary}

\subsubsection{Abelianization of regular Lie algebroids}
We are now ready to look into the abelianization of any regular Lie algebroid. 

\begin{proposition}
    \label{prop:regular:abelianization}
	Let $\A$ be a regular Lie algebroid with isotropy $\g_M$. Then $\A$ has an abelianization if and only if 
	$\overline{[\g_M,\g_M]}\subset \A$ is a vector subbundle. Moreover, if $\A\simeq T\F\ltimes \g_M$ for a quadruple $(\g_M,T\F,\nabla, \Omega)$ then $\nabla$ and $\Omega$ induce a $T\F$-connection and a 2-form on 
    \[ \g^{\ab}_M:=\g_M/\overline{[\g_M,\g_M]}\] and
    \[ \A\ab\simeq T\F\ltimes \g^{\ab}_M. \]
\end{proposition}

\begin{proof}
Assume first that $\overline{[\g_M,\g_M]}\subset \A$ is a vector subbundle. From Proposition \ref{prop:half:main:thm} we already know that $\A$
has abelianization
\[ \A\ab=\A/\overline{[\g_M,\g_M]}. \]
If we assume that $\A\simeq T\F\ltimes \g_M$ for a quadruple $(\g_M,T\F,\nabla, \Omega)$, it follows from Lemma \ref{lem:connection:commutator} that $\nabla$ preserves $\overline{[\g_M,\g_M]}$. By Proposition \ref{prop:extend}, $\nabla$ and $\Omega$ induce a $T\F$-connection and a 2-form on $\g^{\ab}_M:=\g_M/\overline{[\g_M,\g_M]}$ for which we have
\[ \A\ab\simeq T\F\ltimes \g^{\ab}_M. \]

Conversely, suppose $\A$ has abelianization $p:\A\to\A^{\ab}$. We claim that $\g_M$ also admits an abelianization, so by Proposition  \ref{prop:Lie:algebra:bundle} we conclude that $\overline{[\g_M,\g_M]}\subset \A$ is a vector subbundle.

It remains to prove the claim. For that, observe that we have a commutative diagram with exact rows
\[\begin{tikzcd}
		0 & \g_M & \A & T\F & 0 \\
		0 & \g_M^{\ab} & \A^{\ab} & T\F &0
		\arrow[from=1-1, to=1-2]
		\arrow[from=1-2, to=1-3]
		\arrow["\rho", from=1-3, to=1-4]
		\arrow[from=1-4, to=1-5]
		\arrow[from=2-1, to=2-2]
		\arrow[from=2-2, to=2-3]
		\arrow["\rho^{\ab}", from=2-3, to=2-4]
		\arrow[from=2-4, to=2-5]
		\arrow["p_{\g}",from=1-2, to=2-2]
		\arrow["p",from=1-3, to=2-3]
		\arrow["\id",from=1-4, to=2-4]
		,
\end{tikzcd}\]
where $\g\ab_M$ is the isotropy bundle of $\A\ab$ and $p_{\g}=p|_{\g_M}$. If $\g_M$ does not admit an abelianization, then by Proposition \ref{prop:Lie:algebra:bundle}, we know that $\overline{[\g_M,\g_M]}$ does not have a constant rank. As we saw in the proof of Lemma \ref{lem:dense}, we can find an open $O\subset M$ such that 
\[ 
\rank([\g_M,\g_M]|_O)=k<n=\max_{x\in M}\left(\rank\overline{[\g_M,\g_M]}_x\right).
\]
Now consider the Lie algebroid $\A|_O$ and the quotient map
\[ \phi:\A|_O\to \A|_O/[\g_M,\g_M]|_O,\]
where the latter is an abelian Lie algebroid. Since $\phi$ is surjective but $\A\ab$ is of lower rank, $\phi$ cannot be factored by $p:\A\to\A\ab$, which contradicts $\A\ab$ being the abelianization of $A$. Thus $\g_M$ must have an abelianization, as claimed.
\end{proof}

\subsection{Non-regular Lie algebroids}

We consider now arbitrary, possibly non-regular, Lie algebroids. Observe that the definition of the abelianization shows that if $\A\to M$ has abelianization $\A\ab$, then for any open set $U\subset M$ the restriction $\A|_U$ has abelianization $\A\ab|_U$. The following proposition gives a partial converse.

\begin{proposition}\label{pr:restrict}
	Let $p:\A\rightarrow \A\ab$ be a surjective morphism of Lie algebroids over the identity, where $\A\ab$ is abelian. If there exists a dense open $O\subset M$ such that $\A\ab|_O$ is the abelianization of $\A|_O$, then $\A\ab$ is the abelianization of $\A$.
\end{proposition}

\begin{proof}
	Since $p$ is surjective, $\ker(p)$ is a subbundle of $\A$. Let $\phi:\A|_U\rightarrow\B$ be a morphism, where $\B$ is abelian. Since $\A\ab|_O$ is the abelianization of $\A|_O$, we have that $\ker(p|_{U\cap O})\subset \ker(\phi|_{U\cap O})$. It follows that 
    \[ \ker(p|_U)=\overline{\ker(p|_{U\cap O})}\subset\overline{\ker(\phi|_{U\cap O})}\subset \ker(\phi),\] 
    where we use that $O$ is an open dense set and that $\ker(p)$ is a subbundle (here the closures are in $\A|_U$). Thus there is a unique induced algebroid morphism $\tilde{\phi}:\A\ab|_U\rightarrow\B$ such that $\tilde{\phi}\circ p=\phi$. So $\A\ab$ is the abelianization of $\A$. 
\end{proof}
Now we are ready for the main theorem.

\begin{theorem}
\label{th:main}
	A Lie algebroid $\A$ has an abelianization if and only if $\overline{[\g_M,\g_M]}\subset \A$ is a vector subbundle.
\end{theorem}

\begin{proof}
	We already know that if $\overline{[\g_M,\g_M]}\subset\A$ is a vector subbundle then $\A$ admits an abelianization (cf.~Proposition \ref{prop:half:main:thm}). 
 
    For the converse, suppose $\A$ admits an abelianization  $p:\A\to\A\ab$.  Take $U_k$ and $O_k$ as in the proof of Lemma \ref{lem:dense}. Applying \ref{prop:regular:abelianization} to $\A|_{U_k}$, it follows from Lemma \ref{lem:dense} that
    \[ \ker(p_x)=[\g_x,\g_x], \quad \forall x\in O_k. \]
    The union $O=\cup_k O_k\subset M$ is an open dense where 
    \[ \ker(p|_{O})=[\g_M,\g_M]|_{O}. \]
    Since $\ker(p)$ is a subbundle, we must have $\overline{[\g_M,\g_M]}=\ker(p)$, so the result follows.
\end{proof}

\subsection{Functoriality}
One can use the previous results to show that the abelianization behaves functorially relative to several operations with Lie algebroids.

We observed before that, as a consequence of the definition, if a Lie algebroid $\A$ admits an abelianization $\A\ab$ then upon restriction to an open set $U$ one has
\[ (\A|_U)\ab=\A\ab|_U. \]
We can generalize this result to pullback algebroids. For that, we recall that if $\A\to N$ a Lie algebroid and $f:M\to N$ is a map transverse to the anchor $\rho_\A$, then the pullback Lie algebroid $f^!\A\to M$ has supporting vector bundle
\[ f^!\A:=\A\times_{TN}TM=\{(a,v)\in \A\times TM:\rho(a)=\dd f(v)\}. \]
Its anchor is the projection $\pr_{TM}$ and the Lie bracket is given by
\[ [(s_1,X_1),(s_2,X_2)]_{f^!\A}:=([s_1,s_2]_\A,[X_1,X_2]),\]
for any $s_i\in\Gamma(\A)$, $X_i\in\X(M)$, and extended to arbitrary sections so that the Leibniz identity holds. This makes the bundle map
		\[\begin{tikzcd}[ampersand replacement=\&]
			f^!\A \arrow[r, "\pr_\A"]\arrow[d] \& \A\arrow[d]\\
			M\arrow[r, "f"]\& N
		\end{tikzcd}\]
a Lie algebroid morphism. Moreover, if $\phi:\A\to\B$ is a Lie algebroid morphism covering the identity, then there is a morphism between the pullbacks giving a commutative diagram
	\[\begin{tikzcd}[ampersand replacement=\&]
			f^!\A \arrow[r, "f^!\phi"]\arrow[d,"\pr_\A"] \& f^!\B\arrow[d,"\pr_\B"] \\
			\A \arrow[r, "\phi"] \& \B
    \end{tikzcd}\]

\begin{proposition}
    	Let $f: M\rightarrow N$ be a submersion and $\A\to N$ a Lie algebroid with abelianization $\A\ab$. Then $f^!\A$ admits an abelianization and one has
        \[ (f^!\A)\ab\simeq f^!(\A\ab). \]
\end{proposition}

\begin{proof}
    Since the image of a submersion is an open set, we can restrict $\A$ to $f(M)$. Hence, without loss of generality, we can assume that $f$ is surjective. 
    
    Let $p:\A\to\A\ab$ be an abelianization of $\A$. Then $f^!p: f^!\A\to f^!\A\ab$ is a surjective morphism onto an abelian algebroid. Also, if $\g_N$ is the isotropy bundle of $\A$, then the isotropy bundle of $f^!\A$ is
    \[ \g_M:=\{(a,0_x)\in\A\times TM:\rho(a)=0_{f(x)}\}=f^*\g_N. \]
    It follows that
    \[ [\g_M,\g_M]=f^*[\g_N,\g_N]. \]
    By Lemma \ref{lem:dense}, there is an open dense set $O\subset N$ where 
    \[ \ker(p)|_O=[\g_N,\g_N]|_O. \] 
    Since $f$ is a surjective submersion, $O':=f^{-1}(O)$ is an open dense set in $M$ where we have 
    \[ \ker(f^!p)|_{O'}=[\g_M,\g_M]|_{O'}. \] 
    This implies that 
    \[ \overline{[\g_M,\g_M]}=\ker(f^!p),\]
    so $\overline{[\g_M,\g_M]}\subset f^!\A$ is a vector subbundle. By Proposition \ref{prop:half:main:thm}, $f^!\A$ has abelianization
    \[ (f^!\A)\ab=(f^!\A)/\overline{[\g_M,\g_M]}=f^!(\A/\overline{[\g_N,\g_N]})=f^!(\A\ab). \]
\end{proof}

In the previous proof the assumption that $f:M\to N$ is a submersion was used to guarantee that the preimage of a dense open subset $O\subset f(M)$ is an open dense subset of $M$. This property still holds if $f:M\to N$ is a surjective map transverse to a (non-singular) foliation $\F$ of $N$ and the open dense set $O\subset M$ is \emph{saturated}. Hence, for regular Lie algebroids we have the following result.

\begin{proposition}
\label{prop:regular:transverse}
    	Let $\A\to N$ be a regular Lie algebroid and let $f: M\rightarrow N$ be a surjective map transverse to the anchor $\rho_\A$. If $\A$ has an abelianization $\A\ab$, then $f^!\A$ admits an abelianization and one has
        \[ (f^!\A)\ab\simeq f^!(\A\ab). \]
\end{proposition}

\begin{proof}
    Since $f$ is transverse to the anchor $\rho_\A$, the pullback algebroid $f^!\A$ is well defined. The proof of the previous proposition holds word-by-word if one can show that $f^{-1}(O)$ is open and dense in $M$ where $O\subset N$ is the dense, saturated, open set given by Lemma \ref{lem:dense}.

    To see this, cover $M$ by foliated charts $(U_i,\phi_i)$ for $T\F=\mathrm{im}(\rho_\A)$. So $\phi_i:U_i\to \R^q$ is a submersion such that the plaques of $\F$ in $U_i$ are the level sets $\phi_i^{-1}(p)$. Since 
    $f: M\rightarrow N$ is transverse to $\F$, the composition
    \[ f_i:=\phi_i\circ f:f^{-1}(U_i)\to \R^q, \]
    is a submersion. Since $O$ is saturated, we see that
    \[ f^{-1}(O\cap U_i)=f_i^{-1}(\phi_i(O\cap U_i)) \]
    is open and dense in $f^{-1}(U_i)$. Since $f$ is surjective and $N=\bigcup_I U_i$, it follows that $f^{-1}(O)$ is open and dense in $M$, as claimed.
\end{proof}

We provide an example illustrating the assumptions in the previous two results.

\begin{example}
	Let $M=\T^2\simeq(\R/2\pi\Z)^2$ and consider the Lie algebroid $\A\to \T^2$ with underlying bundle the trivial rank 3 vector bundle, anchor
    \[ \rho(e_1):=\frac{\partial}{\partial \theta^1},\quad \rho(e_2)=\rho(e_3):=0,\]
    and Lie bracket defined by
    \[ [e_2,e_3]_{(\theta^1,\theta^2)}:=\sin(\theta^2)\, e_2,\quad [e_1,e_2]=[e_1,e_3]:=0. \]
    This is a regular Lie algebroid with isotropy bundle 
    \[ \g_M=\R e_2\oplus \R e_3, \]
    and commutator bundle
    \[ [\g_M,\g_M]_{(\theta^1,\theta^2)}= \begin{cases}
    \R e_2, & \text{if }\theta^2\ne 0\\
    0, & \text{if }\theta^2=0.
    \end{cases}\]
    This has closure $\overline{[\g_M,\g_M]}=\R e_2$, a vector subbundle of $\A$, so the abelianization $\A\ab$ exists (and has rank 2). 
    
    Now consider the surjective map $f:\T^2\to\T^2$, $f(\theta^1,\theta^2)=(s(\theta^1),\theta^2)$, where $s:\S^1\to\S^1$ is a surjective map which equals $0$ on some open interval $I\subset \T$. The map $f$ is transverse to the anchor, so by Proposition \ref{prop:regular:transverse} the abelianization of $f^!\A$ exists. 
\end{example}

\subsection{Application to Poisson geometry}

Recall that a Poisson manifold is a pair $(M,\pi)$ consisting of a smooth manifold $M$ and a bivector $\pi\in\mathfrak{X}^2(M)$ whose Schouten bracket vanishes: 
\[ [\pi,\pi]=0. \]
A Poisson structure can also be defined as a Lie bracket on the space of smooth functions which is also a biderivation. We refer to \cite{CFM22} for background in Poisson geometry. 

\begin{example}[Linear Poisson structure]
	Let $\g$ be a Lie algebra. Fix a basis $\{e_i\}$ for $\g$ and denote by $\{\sigma_i\}$ the dual of $\g^*$.  Then $(\g^*,\pi_{\g})$ is a Poisson manifold, where $\pi_{\g}$ is defined by
    $$\pi_{\g}:=\frac{1}{2}\sum_{i,j,k}c^{ij}_k\sigma_k\frac{\partial}{\partial \sigma_i}\wedge\frac{\partial}{\partial \sigma_j},$$
    with $c^{ij}_k$ denoting the structure constants relative to the basis $\{e_i\}$. 
\end{example}

It is well-known that the cotangent bundle of a Poisson manifold $(M,\pi)$ has an induced Lie algebroid structure. The anchor is given by the contraction map $\pi^{\#}:T^*M\to TM,\alpha\mapsto i_{\alpha}\pi$ and the Lie bracket on the space of one forms $\Omega^1(M)$ is given by:
$$[\alpha,\beta]_{\pi}:=\Lie_{\pi^{\sharp}\alpha}(\beta)-\Lie_{\pi^{\sharp}\beta}(\alpha)-d(\pi(\alpha,\beta)).$$
We call $(T^*M,[\cdot,\cdot]_{\pi},\pi^{\#})$ the \emph{cotangent Lie algebroid} associated to $(M,\pi)$. 

Via the cotangent Lie algebroid, all the notions associated to Lie algebroids also arise for Poisson manifolds. In particular, $\mathrm{im}\,(\pi^{\#})\subset TM$ integrates to a singular foliation, and for any $x\in M$ one has the isotropy Lie algebra $\ker(\pi^{\#}_x)\subset T_x^*M$. This isotropy Lie algebra coincides with $(\mathrm{im}\,\pi^{\#})^0$, the conormal space to the foliation.

\begin{proposition}\label{prop:pois:sing}
    A Poisson manifold with 
	non-abelian cotangent Lie algebroid does not admit an abelianization.
\end{proposition}
\begin{proof}
Recall first that for a regular Poisson manifold, the isotropy Lie algebras are all abelian. In particular, this also holds for the regular set of any Poisson manifold $(M,\pi)$, providing a dense open set $M^{\textrm{reg}}\subset M$ where the isotropy Lie algebras are abelian. 

Now, assume that $T^*M$ is non-abelian. Then there exists $x\in M$ such that $[\g_x,\g_x]\neq 0$. On the other hand, $[\g,\g]=0$ on the open set $M^{\textrm{reg}}$. Hence $\overline{[\g,\g]}$ is not a subbundle and by Theorem \ref{th:main}, $T^*M$ admits no abelianization.
\end{proof}
\begin{example}
	Consider a linear Poisson structure $(\g^*, \pi_{\g})$ corresponding to a Lie algebra $\g$. The isotropy Lie algebra at $x=0$ is $\g$, so it follows that the cotangent Lie algebroid is abelian if and only if $\g$ is abelian. So by Proposition \ref{prop:pois:sing}, $T^*\g^*$ has abelianization iff $\g$ is abelian, i.e., iff $\pi_{\g}\equiv0$.
\end{example}

Note that there are many examples of non-regular Poisson structures whose cotangent algebroid is abelian. A simple example is the Poisson structure on $\R^2$ defined by
    \[ \pi=(x^2+y^2)\frac{\partial}{\partial x}\wedge\frac{\partial}{\partial y}. \]
    
\begin{remark}
    Proposition \ref{prop:pois:sing} also holds for Dirac structures and, more generally, whenever one has a geometric structure defining a Lie algebroid whose regular part is forced to be abelian.    
\end{remark}

\section{Abelianization of Lie groupoids}\label{sec:gp}
\subsection{Definition of abelianization}
\subsubsection{Backgrounds on groupoids}

Recall that a groupoid $\G$ is a small category in which every arrow is invertible. We will use the following notations:
	\begin{itemize}
		\item $\G_1$ and $\G_0$ denote the sets of arrows and of objects;
		\item $s$ and $t$ denote the source and target maps;
        \item $m$, $i$ and $u$ denote the multiplication, the inverse and the unit maps, respectfully.
	\end{itemize}
 
For a groupoid $\G$, when it does not cause confusion, we will denote the set of arrows by $\G$ and the set of objects by $M$.  
\begin{definition}
	A Lie groupoid is a groupoid where $\G_1$ and $\G_0$ are smooth manifolds, all the structure maps are smooth and the source and target maps are submersions.
\end{definition}

As usual we do not assume $\G_1$ to be Hausdorff. We refer to \cite{CFM22,MK87} for details and background on Lie groupoids.
The proof of the following standard result can be found in \cite{MK87}.

\begin{proposition}\label{prop:normal}
	Let $\G$ be a Lie groupoid. If $\HH$ is a closed normal Lie subgroupoid of $\G$, then $\G/\HH$ is a Lie groupoid.
\end{proposition}
For $x\in M$ we have the \emph{isotropy group} at $x$, defined by $\G_x=s^{-1}(x)\cap t^{-1}(x)$. For a Lie groupoid, this is in fact a Lie group. There is also the \emph{orbit} through $x$, defined by $O_x=t(s^{-1}(x))$. For a Lie groupoid, this is an immersed submanifold of $M$. The smooth structure is induced by the principal bundle $\G_x\circlearrowright s^{-1}(x)\xrightarrow{t} O_x$. When the dimension of the orbits is constant we call the groupoid \emph{regular}. When there is only a single orbit, we say the groupoid is \emph{transitive}. We call a Lie groupoid \emph{source connected} if the source fibres are connected, and we call it \emph{source 1-connected} if they are in addition simply connected.
A groupoid is called \emph{abelian} if all the isotropy groups are abelian.

\subsubsection{Abelianization}

For the following definition we assume that one is working in a topological category.

\begin{definition}
	The abelianization of a groupoid $\mathcal{G}\rightrightarrows M$ consists of:
	\begin{itemize}
		\item an abelian groupoid $\mathcal{G}^{ab}\rightrightarrows M$;
		\item a surjective morphism $p:\mathcal{G}\rightarrow\mathcal{G}^{ab}$ covering $Id_M$; 
	\end{itemize}
	such that for any open subset $U\subset M$, any abelian groupoid $\mathcal{H}\rightrightarrows N$ and morphism $\psi:\mathcal{G}|_U\rightarrow\mathcal{H}$, we have: 
	\begin{center}
		$\begin{tikzcd}[ampersand replacement=\&]
			\mathcal{G}|_U \arrow[r, "\psi"]\arrow[d,"p"] \& \mathcal{H} \\
			\mathcal{G}\ab|_U\arrow[dashed]{ur}[swap]{\exists ! \tilde{\psi}}
		\end{tikzcd}.$
	\end{center}
\end{definition}
Just as in the case of Lie algebroids, the abelianization is unique up to isomorphisms, if it exists.
The abelianization of groupoids is heavily dependent on the category we are working in. The existence of abelianizations differs greatly between different categories, and even if it exists in two categories, the abelianizations themselves can be different. The following examples illustrate these phenomena.
\begin{example}
	A topological groupoid $\G$ always has an abelianization, namely $\mathcal{G}^{ab}=\mathcal{G}/(\mathcal{G}_M,\mathcal{G}_M)$, where $\G_M$ denotes the bundle of isotropy groups of $\G$. In particular, the abelianization of a topological group is $G/(G,G)$.
\end{example}
\begin{example}
	A Lie group always has an abelianization $G^{ab}=G/\overline{(G,G)}$.
\end{example}  

So we see that already for a group, the abelianizations can differ between categories, since in general $(G,G)\subsetneq\overline{(G,G)}$.

\begin{example}\cite{bourbaki1989lie}
	Consider $G=(\widetilde{SL(2,\R)}\times\S^1)/\Z$, where $\Z$ is embedded as the covering group of $\widetilde{SL(2,\R)}$, the universal cover of $SL(2,\R)$, and  as a non-discrete subgroup of $\S^1$. Note that since $(SL(2,\R),SL(2,\R))=(SL(2,\R))$, the commutator $(G,G)$ is actually dense in $G$, thus $\overline{(G,G)}=G\neq(G,G)$.
\end{example}

\subsection{Abelianization of Lie groupoids}

A Lie groupoid does not always admit a smooth abelianization. This was already observed in \cite{contreras2018genus} where the following example was discussed.

\begin{example}
	Consider the action Lie groupoid $\mathcal{G}=SO(3)\times\mathbb{R}^3\rightrightarrows\mathbb{R}^3$ associated with the usual action $SO(3)$  on $\R^3$ by rotations. The isotropy group $\G_0$ at $0$ is $SO(3)$, and so is its commutator $(\G_0,\G_0)$. Any potential abelianization will need to quotient out $(\G_0,\G_0)$ and thus can not recover the natural map from this groupoid to the pair groupoid $\R^3\times\R^3$ away from $0$.
\end{example}

Although a Lie groupoid may not admit a admit a smooth abelianization, for \emph{transitive} Lie groupoids, one can always find such an abelianization.

\begin{proposition}\cite{contreras2018genus}
	A transitive Lie groupoid $\G$ has abelianization $\G\ab=\G/\overline{(\G_M,\G_M)}$.
\end{proposition}

Let $\overline{(\G_M,\G_M)}$ be the closure of the commutator $(\G_M,\G_M)$ in $\G$. We let 
$$\overline{(\G_M,\G_M)}^s:=\bigcup\overline{(\G_x,\G_x)}^s$$ 
be the union of all ``fiberwise closure''  of $(\G_x,\G_x)$. By ``fiberwise closure'', we mean that $\overline{(\G_x,\G_x)}^s$ is the closure of the commutator of $\G_x$ in the source fiber. In general, these two closures are not the same and we have $\overline{(\G_x,\G_x)}^s\subset \overline{(\G_x,\G_x)}$.

\begin{example}
	Let us look at the groupoid version of Example \ref{ex:basic}. There $\overline{(\G_M,\G_M)}$ is the trivial line bundle over $\R$, while $\overline{(\G_M,\G_M)}^s$ is $0$-dimensional at $0$. 
\end{example}

It is an immediate consequence of Proposition \ref{prop:normal} that if $\overline{(\G_M,\G_M)}$ is a normal Lie subgroupoid, then $\G/\overline{(\G_M,\G_M)}$ is the abelianization of $\G$. In this respect, we have the following observation.

\begin{lemma}\label{lem:clo:nor}
	If $\overline{(\G_M,\G_M)}$ is a subgroupoid of $\G$, then $\overline{(\G_M,\G_M)}$ is a normal subgroupoid of $\G$. 
\end{lemma}

\begin{proof}
	Consider $h\, g\, h^{-1}$ where $g\in\overline{(\G_M,\G_M)}$ and $h\in\G$ such that  $s(h)=t(g)$. Let $g_i\in (\G_M,\G_M)$ be a sequence such that $g_i\rightarrow g$. We can pick $h_i\to h$ a sequence in $\G$ such that $s(h_i)=t(g_i)$. Now we have $h_i\, g_i\, h_i^{-1}\rightarrow h\, g\, h^{-1}$. A simple computation shows that $h_i\,g_i\, h_i^{-1}\in (\G_M,\G_M)$, which implies that $h\,g\, h^{-1}\in\overline{(\G_M,\G_M)}$ and thus that $\overline{(\G_M,\G_M)}$ is normal in $\G$.
\end{proof}

\begin{proposition}
	Given a Lie groupoid $\G$, if $\overline{(\G_M,\G_M)}^s$ is a closed submanifold, then $\G$ has abelianization $\G/\overline{(\G_M,\G_M)}^s$. 	
\end{proposition}
\begin{proof}
	If $\overline{(\G_M,\G_M)}^s$ is closed, then $\overline{(\G_M,\G_M)}^s=\overline{(\G_M,\G_M)}$. Since it is contained in the isotropy and on each source fiber it is a subgroup, it is a subgroupoid. Thus by lemma \ref{lem:clo:nor}, it is a normal Lie subgroupoid. So the quotient $\G/\overline{(\G_M,\G_M)}^s$ is the abelianization of $\G$.
\end{proof}

Note that the condition in the previous proposition is not a necessary condition for existence of an abelianization. For instance, the groupoid version of Example \ref{ex:basic} provides an example of a Lie groupoid which has an abelianization but where the fiberwise closure $\overline{(\G_M,\G_M)}^s$ is not a submanifold.

\subsection{Abelianization of diffeological groupoids}

We now consider abelianization of groupoids in the category of diffeological spaces. We will see that the flexibility allowed by diffeologies yields existence of abelianizations for a wide class of groupoids.

\subsubsection{Diffeology}
  In this section we describe only the definitions and results strictly necessary to express our results. For a detailed introduction to diffeologies we refer to  \cite{diffeo}.

In the following definition, by a \emph{parameterization} of a set $X$ we mean any map from any open $U\subset\R^n$ to $X$.

\begin{definition}
	A diffeological space is a non-empty set $X$ together with a set $\D$ of parameterizations of $X$, such that:
	\begin{itemize}
		\item For all $x\in X$ and all open set $U\subset\R^n$, the constant parameterization $\mathbf{x}:U\to X,\, r\mapsto x$ is in $\D$.
		\item If $P:U\rightarrow X$ is a parametrization and for any $r\in U$ there exists open neighborhood $V$ of r such that $P|_V\in \D$, then $P$ is in $\D$.
		\item For any smooth parametrization $P:U\rightarrow X$ in $\D$, any open subset $V\subset \mathbb{R}^k$ for some $k\in\mathbb{N}$, and any smooth map $F\in C^{\infty}(V,U)$, $P\circ F$ is in $\D$.
	\end{itemize}
	We call the elements of $\D$ the \emph{plots} of the diffeological space.
\end{definition}
A map $f:X\rightarrow X'$ of diffeological spaces is called  \emph{smooth} if for each plot $P$ of $X$, $f\circ P$ is a plot of $X'$. A \emph{diffeological groupoid} is a groupoid internal to the category of diffeological spaces.

\begin{example}
	Given a diffeological space $(X,\D)$ and a surjective map $q:X\rightarrow \tilde{X}$ one obtains a quotient diffeology $\D^q$ on $\tilde{X}$ as follows: $P: U\rightarrow\tilde{X}$ belongs to $\D^q$ if and only if for any $r\in U$ there exists an open neighborhood $V$ of $r$ such that $P|_V=q\circ Q$ for some $Q\in \D$.
\end{example}
One can check by definition that the quotient diffeology satisfies the universal property, i.e., for any diffeological space $Z$ and any map $f:\tilde{X}\rightarrow Z$, $f$ is smooth if and only if $f\circ q$ is smooth.
\begin{example}
	Given a diffeological space $(X,\D)$ and a subset $Y\subset X$ with inclusion $j:Y\rightarrow X$, one obtains a subset diffeology $\D_Y$ (often denoted as $j^*(\D)$). on $Y$ as follows: $P\in \D_Y$ if and only if $\mathrm{img}(P)\subset Y$ and $j\circ P\in \D$.
\end{example}

Next we recall that for any diffeological space, there is a natural topology associated to it.
    \begin{definition}
	For a diffeological space $(X,\D)$, its \emph{$\D$-topology}  $\mathcal{T}_{\D}$ is the finest topology making all the plots continuous, i.e.:
	$$\mathcal{T}_{\D}:=\{O\subset X:P^{-1}({O}) \text{ is open}, \forall P\in\D\}.$$
\end{definition}
It is immediate from the Definiton that smooth maps between diffeological spaces are continuous with respect to their $\D$-topology. In the sequel, we will work with this topology by default.

\subsubsection{Abelianization}

\begin{proposition}
	Given a diffeological groupoid $\G$ the quotient $\G/(\G_M,\G_M)$ is the abelianization of $\G$. 
\end{proposition}

\begin{proof}
	We know that $\tilde{\G}:=\G/(\G_M,\G_M)$ is a groupoid, so we only need to show that all its structure maps are smooth. The identity map $\tilde{u}=q\circ u$ is clearly smooth. The rest follows from the fact that the structure maps of $\G$ and $\tilde{\G}$ fit into commutative diagrams
$$\begin{tikzcd}		\G_1&\G_1\\
		\tilde{\G_1}&\tilde{\G_1}
		\arrow["i",from=1-1,to=1-2]
		\arrow["q",from=1-1,to=2-1]
		\arrow["q",from=1-2,to=2-2]
		\arrow["\tilde{i}",from=2-1,to=2-2]
	\end{tikzcd}\qquad
	\begin{tikzcd}
		\G_2&\G_1\\
		\tilde{\G_2}&\tilde{\G_1}
		\arrow["m",from=1-1,to=1-2]
		\arrow["q\times q",from=1-1,to=2-1]
		\arrow["q",from=1-2,to=2-2]
		\arrow["\tilde{m}",from=2-1,to=2-2]
	\end{tikzcd}$$
	$$\begin{tikzcd}
		\G\\
		\tilde{\G_1}&\G_0
		\arrow["s",from=1-1,to=2-2]
		\arrow["q",from=1-1,to=2-1]
		\arrow["\tilde{s}",from=2-1,to=2-2]
	\end{tikzcd}\qquad \begin{tikzcd}
		\G\\
		\tilde{\G}&\G_0
		\arrow["t",from=1-1,to=2-2]
		\arrow["q",from=1-1,to=2-1]
		\arrow["\tilde{t}",from=2-1,to=2-2]
	\end{tikzcd}$$

To show that $\tilde{\G}$ satisfies the universal property, let $U\subset M$ be an open set. First note that $\G|_U$ with the subset diffeology carries a natural diffeological groupoid structure. Given any abelian diffeological groupoid $\HH$ and any morphism $\psi:\G|_U\to \HH$, since $(\G|_U,\G|_U)\subset\ker(\psi)$, there is a unique morphism $\tilde{\psi}:\tilde{\G}|_U\rightarrow \HH$ such that $\psi=\tilde{\psi}\circ q|_U$. It is easy to check that the subset diffeology $\D_{\tilde{\G}|_U}$ on $\tilde{\G}|_U$ coincides with the quotient diffeology on $\G|_U/(\G|_U,\G|_U)$. Thus by universal property of quotients, $\tilde{\psi}$ is smooth since $\psi=\tilde{\psi}\circ q|_U$ is smooth.
\end{proof}

We consider now the following subcategory. It plays, e.g., an important role in the integration problem for Lie algebroids -- see \cite{villatoro2023integrability}.

\begin{definition}
	A surjective smooth map $f:(X,\D)\rightarrow (X',\D')$ is a \emph{subduction} if the following property holds: a plot $P:U\rightarrow X'$ belongs to $\D'$ if and only if for any $r\in U$ there exists an open neighborhood $V$ of $r$ such that $P|_V=f\circ Q$ for some $Q\in \D$.
\end{definition}
The source and target maps of a diffeological groupoid are automatically subductions.

\begin{definition}
	A smooth map $f:X\rightarrow X'$ is \emph{locally subductive} at $x\in X$ if for any plot $P$ of $X'$ such that $P(0)=f(x)$, there exist an open neighborhood $V$ of 0 and a plot $Q:V\rightarrow X$ of $X$ such that $Q(0)=x$ and $f\circ Q=P|_V$. We call $f$ a \emph{local subduction} if it is locally subductive at all $x\in X$.
\end{definition}

\begin{example}
	Let $f:\R\rightarrow\R$ be such that $f(0)=1$ and $f$ vanishes when $|x|>1$. Then $g:\R^2\rightarrow\R=f(x)\cdot y$ is a subduction but not a local subduction.
\end{example}
A diffeological groupoid is called \emph{locally subductive} if the source and target maps are local subductions.
\begin{proposition}
	A locally subductive diffeological groupoid $\G$ always has abelianization $\G\ab=\G/(\G_M,\G_M)$, which is also a locally subductive diffeological groupoid.
\end{proposition}

\begin{proof}
	Let $[g]\in \G_1\ab$ and $P$ be a plot of $\G_0=\tilde{\G_0}$ such that $P(0)=s\ab([g])$. Then since $s$  is locally subductive, we have open neighborhood $V$ of 0 and $Q$ in plot of $\G_1$ such that $Q(0)=g$ and $s\circ Q=P|_V$. Now take $\tilde{Q}=q\circ Q$, which will be a plot of $\G_1\ab$. We have $\tilde{Q}(0)=[g]$ and $\tilde{s}\circ \tilde{Q}=s\circ Q=P|_V$. Thus $\tilde{s}$ is a local subduction.
\end{proof}
\bibliographystyle{alpha}
\nocite{*}
\bibliography{abelianization.bib}
\end{document}